\pdfoutput=1
\documentclass[11pt]{amsart}

\usepackage[utf8]{inputenc}
\usepackage[T1]{fontenc}
\usepackage[text={445pt,630pt}, centering, a4paper, marginparwidth=2cm]{geometry}
\usepackage[hidelinks,pdfusetitle]{hyperref}
\usepackage[activate={true,nocompatibility},final,tracking=true,kerning=true,factor=1100,stretch=10,shrink=10]{microtype}
\usepackage{braket}    
\usepackage{amsthm}    
\usepackage{amssymb}   
\usepackage{mathrsfs}  
\usepackage[shortlabels]{enumitem}
\usepackage{caption}

\usepackage{booktabs,siunitx}

\usepackage{bbold}
\usepackage{bm}

\theoremstyle{plain}
\newtheorem*{theo*}{Theorem}
\newtheorem{theorem}{Theorem}[section]

\newtheorem{proposition}[theorem]{Proposition}
\newtheorem{lemma}[theorem]{Lemma}

\theoremstyle{remark}
\newtheorem{remark}[theorem]{Remark}
\newtheorem{example}[theorem]{Example}

\newcommand{\cf}{\ensuremath{\mathscr{F}}}

\newcommand{\ind}[1]{\mathbb{1}_{#1}}
\newcommand{\un}{\mathbb{1}}
\newcommand{\zero}{\mathbb{0}}

\newcommand{\cl}{\ensuremath{\mathscr{L}^\infty}}

\newcommand{\kkk}{{\ensuremath{\bm{\kk}}}}

\usepackage[backend=biber,
	    url=false,
	    doi=false,
	    isbn=false,
	    date=year,
	    maxbibnames=99,
	    giveninits,
	    sortcites=true,
	    sorting=nty,
	    style=numeric]{biblatex}
\renewbibmacro{in:}{}
            
\date{\today}

\author{Jean-François Delmas}
\address{Jean-François Delmas,
  CERMICS, \'{E}cole des Ponts, France}
\email{jean-francois.delmas@enpc.fr}

\author{Dylan Dronnier}
\address{Dylan Dronnier,
  Université de Neuchâtel, Switzerland}
\email{dylan.dronnier@unine.ch}

\author{Pierre-André Zitt}
\address{Pierre-André Zitt, LAMA, Université Gustave Eiffel, France}
\email{pierre-andre.zitt@univ-eiffel.fr}

\newcommand{\norm}[1]{\left\lVert\,#1\,\right\rVert}

\newcommand{\R}{\ensuremath{\mathbb{R}}}

\newcommand{\kk}{\ensuremath{\mathrm{k}}}

\newcommand{\Tinf}{\ensuremath{\mathcal{T}}}

\newcommand{\etae}{\ensuremath{\eta^\mathrm{equi}}} 
\newcommand{\etauc}{\ensuremath{\eta^\mathrm{uni}_\mathrm{crit}}}
\newcommand{\etau}{\ensuremath{\eta^\mathrm{uni}}} 
\newcommand{\etao}{\ensuremath{\eta^\mathrm{opt}}}

\title[Vaccinating according to the maximal endemic
equilibrium]{Vaccinating according to the maximal endemic equilibrium
  achieves herd immunity}

\begin{document}

\begin{abstract}
  We  consider  the  simple  epidemiological SIS  model  for  a  general heterogeneous
  population introduced by Lajmanovich and Yorke (1976) in finite  dimension,  and  its
  infinite  dimensional  generalization  we introduced in previous works.  In this model
  the basic reproducing number $R_0$ is given by the spectral radius of an integral
  operator. If $R_0>1$, then  there exists a  maximal endemic equilibrium. In this very
  general  heterogeneous SIS model, we prove that vaccinating  according   to  the 
  profile  of   this  maximal  endemic equilibrium ensures herd immunity. Moreover, this
  vaccination strategy is critical:  the resulting  effective reproduction number  is
  exactly equal to one. As an application, we estimate that if $R_0 = 2$ in an
  age-structured community with mixing rates fitted to social activity, applying this
  strategy would require approximately 29\% less vaccine doses than the strategy which
  consists in vaccinating uniformly a proportion $1 - 1/R_0$ of the population.

  From a dynamical systems point of view,  we prove  that  the  non-maximality of  an 
  equilibrium $g$  is equivalent to its linear instability  in the original dynamics, and
  to the  linear instability  of  the disease-free  state  in the  modified dynamics where
  we vaccinate according to~$g$. 
\end{abstract}

\maketitle

\section{Introduction}

Increasing  the prevalence  of  immunity from  contagious  disease in  a population limits
the circulation of the infection among the individuals who lack  immunity. This so-called 
``herd effect'' plays  a fundamental role  in epidemiology,  for example  it has  had a 
major impact  in the eradication  of  smallpox and  rinderpest  or  the near  eradication 
of poliomyelitis \cite{HerdImmunityFine2011}.  Our aim is to present a targeted 
vaccination  strategy  based   on  the  heterogeneity  of  the infection  spreading in 
the population  which allows  to eradicate  the epidemic.     We    consider    for  
simplicity    the deterministic infinite-dimensional SIS model (with S=Susceptible and
I=Infectious) and the effect of  a perfect vaccine.  However, we take  into account a very
general   model  for the heterogeneous  population   based  on   the
infinite-dimensional model introduced in
\cite{delmas_infinite-dimensional_2020},     that    encompasses     the
meta-population~SIS models  developed by Lajmanovich and  Yorke in their pioneer 
paper~\cite{lajmanovich1976deterministic}   or  SIS   model  on graphs. More   precisely, 
the   probability  $u_t(x)$   of  an individual of  type $x\in \Omega$  to be infected  at
time $t$  is the solution of the (infinite dimensional) ordinary differential equation:
\begin{equation}
  \label{eq:def-sis}
  \partial_t u_t(x) = (1 - u_t(x)) \int_\Omega k(x,y)\, u_t(y) \,
  \mu(\mathrm{d} y) - 
  \gamma(x) u_t(x) \quad\text{for $t\geq 0$ and $x\in \Omega$,} 
\end{equation} 
where $k$ is the transmission rate  kernel of the disease,  $\gamma$ is the recovery  rate 
function and  $\mu(\mathrm{d}  y)$  is the probability  for  an individual  taken  at 
random  to  be  of  type  $y\in  \Omega$;  see Equation~\eqref{eq:SIS2}.

\medskip

In an homogeneous population, the \emph{basic reproduction number} of an infection,
denoted by~$R_0$, is defined as the number of secondary cases one individual  generates on
average  over the course of  its infectious period,   in   an   otherwise   uninfected  
(susceptible)   population. Intuitively,  the  disease should  die  out  if~$R_0<1$ and 
invade  the population  if~$R_0>1$.   For  the heterogenous generalization of many 
classical  models  in epidemiology (including the heterogeneous SIS model), it is still
possible to define a meaningful  basic reproduction number~$R_0$, as the number of
secondary cases generated by a \textit{typical} infectious individual    when    all    
other    individuals   are uninfected and  the threshold  phenomenon occurs
\cite{Diekmann1990}.   In the setting  of~\cite{delmas_infinite-dimensional_2020},  the 
reproduction number  $R_0$ then corresponds to  the spectral  radius of  the next-generation
operator defined as  the  integral operator  associated  to  the  kernel
$k(x,y)/\gamma(y)$.

After a vaccination campaign, let the vaccination strategy $\eta$ denote the (non
necessarily homogeneous) proportion of the \textbf{non-vaccinated} population,  and let
the effective reproduction number $R_e(\eta)$ denote the corresponding reproduction 
number of the non-vaccinated population.
Following~\cite[Section~5.3.]{delmas_infinite-dimensional_2020}, the effective
reproduction number $R_e(\eta)$ is given by the spectral radius of the effective
next-generation operator  defined as the integral operator associated to  the kernel
$k(x,y)\eta(y) /\gamma(y)$,  where $\eta(y)$ is the proportion of individuals of type $y$
which are not vaccinated.

The vaccination  strategy~$\eta$  is called \emph{critical} if $R_e(\eta) =  1$. 
Assuming~$R_0>1$, suppose now that only a proportion~$\etau$  of the population can catch 
the disease, the rest being  perfectly immunized.  An  infected individual will  now only
generate~$\etau  R_0$  new  cases,  since  a  proportion~$1-\etau$  of previously
successful infections will  be prevented.  Therefore, the new \emph{effective reproduction
number} is equal to~$R_e(\etau) = \etau R_0$. This fact led to the recognition by Smith in
1970 \cite{smith_prospects_1970} and Dietz in 1975 \cite{dietz1975transmission} of a
simple threshold theorem: the incidence of an infection declines if the proportion of
non-immune individuals is reduced below~$\etauc = 1/R_0$.  This effect is called
\emph{herd immunity}, and the corresponding proportion~$1-\etauc$ of people that have to
be vaccinated is called the \emph{herd immunity threshold} \cite{smith_concepts_2010,
somerville_public_2016}.

\section{Critical vaccination given by the endemic equilibrium}

However,  herd  immunity  can  also  be  achieved  using  a  non-uniform vaccination
strategy when the population is heterogeneous.  For example, the discussion of    
vaccination control of gonorrhea in~\cite[Section~4.5]{hethcote} suggests that  it  may 
be  better  to prioritize  the  vaccination of  people  that  have already  caught  the
disease: this leads us to consider  a vaccination strategy guided by the equilibrium 
state. For the SIS model in heterogeneous population  with $R_0>1$,
there exists a  maximal endemic equilibrium, say  $\mathfrak{g}$, where $\mathfrak{g}(x)$
represents the fraction of infected people in  the group with feature $x$. In other words,
the function $\mathfrak{g}$  is  the  maximal  $[0,  1]$-valued solution $g$ of:
\begin{equation}\label{eq:endemic}
  (1 - g(x)) \int_\Omega  k(x,y) \, g(y)\, \mu(\mathrm{d} y) =
  \gamma(x) g(x) \quad\text{for  $x\in \Omega$.} 
\end{equation}
 
Let us mention that if there exist isolated subpopulations, it is possible to have other
endemic equilibria, \textit{i.e.}, solutions to Equation~\eqref{eq:endemic} that are not
equal to $0$ for all $x$. Irreducibility conditions on the kernel $k$ ensure however the
uniqueness of the endemic equilibrium
\cite{lajmanovich1976deterministic,delmas_infinite-dimensional_2020}. Consider the
vaccination strategy, denoted by~$\etae$, corresponding to vaccinating a fraction
$\mathfrak{g}(x)$ of people in the group with feature $x$, for all groups. In  our
mathematical framework, this amounts to setting:
\begin{equation}
  \etae(x)=1 - \mathfrak{g}(x) \quad\text{for  $x\in \Omega$.}
\end{equation}
The following  result ensures that this strategy reaches herd immunity; see
Theorem~\ref{theo:main} in Section~\ref{sec:main} for a precise mathematical
statement.

\begin{theo*}
  In  the   heterogeneous  SIS  model  with   non-zero  maximal  endemic
  equilibrium, the vaccination strategy~$\etae$ is critical:
  \begin{equation}
     R_e(\etae)=1.
  \end{equation}
\end{theo*}

Let us stress that implementing the critical vaccination strategy $\etae$ can be achieved
without estimating  the transmission rate kernel  and the recovery rate.

\medskip

The proof of the theorem relies on  the study of the spectral bound of the linearized
operator associated to equation~\eqref{eq:def-sis} near an equilibrium. When  $R_0>1$,
this spectral bound  is non-positive at the maximal  equilibrium and  positive at  all
other  equilibria; see Proposition~\ref{prop:caract-g}~\ref{s(g)<0}. Thus, the
non-maximality of an equilibrium  is equivalent to its linear  instability in the original
dynamics.  We also prove  the  linear instability  of  the disease-free  state  in the 
modified dynamics where we vaccinate according to a non maximal equilibrium; see
Proposition~\ref{prop:caract-g}~\ref{s(1-h)<0}.

\section{Discussion}

We  expect  the  results obtained  here for  the  SIS model  to  be generic,  in  the
sense  that similar behaviours should also be observed in more realistic and complex
models in epidemiology for non-homogeneous populations: when an endemic equilibrium
exists, vaccinating the population according to the maximal endemic profile should protect the
population from the disease.

\medskip

We  refer to~\cite{ddz-theo} for a general framework for cost comparison of vaccination
strategies and the notions of ``best'' and ``worst'' vaccination strategies; see also
\cite{ddz-Re,ddz-cordon,ddz-reg} for further comments and various examples of optimal
vaccinations.

Consider a general cost function~$C$ which measures  the cost for the society  of a 
vaccination strategy  (production and  diffusion).  A simple  and natural  choice is  the
uniform  cost  given  by the overall proportion of vaccinated individuals:
\begin{equation}
  \label{eq:def-cu}
  C (\eta)=\int_\Omega (\un-\eta)\, \mathrm{d}\mu=1- \int_\Omega
  \eta\, \mathrm{d}\mu.
\end{equation}
We have that $C(\etauc)$ is equal to the herd immunity threshold $1- 1/R_0$
while~$C(\etae)$ is equal to  the proportion of people in the endemic state in an SIS
infection $\int_\Omega \mathfrak{g} \, \mathrm{d}\mu$. It is not possible to determine
which strategy is cheaper in general. However, in the following examples, we are able to
compare their costs for mixing structures that are redundant in the epidemiologic
litterature.

\begin{example}[Homogeneous mixing]
  If the population is homogeneous (which corresponds to the one-dimensional
  SIS model where $\Omega$ is a singleton), then the maximal equilibrium
  is constant equal to $1-1/R_0$. It follows that $C(\etauc)=C(\etae)$.
\end{example}

\begin{example}[Proportionate mixing structure with two subpopulations]
  The proportionate mixing  is a classical mixing  structure introduced by
  \cite{HeterogeneityINold1980} and used in many different epidemiological models.  It
  assumes  that the  number of  adequate contacts  between two subpopulations is 
  proportional to the  relative activity levels  of the two subpopulations. Thus
  individuals  in more active subpopulations will have more adequate contacts. Let us
  consider the simple case where there are only two subpopulations. Then the contact
  matrix is given by:
  \[
    K = 
    \begin{pmatrix}
      a^2 & ab \\
      ab & b^2
    \end{pmatrix}
  \]
  where $a$ and $b$ are positive constants that correspond to the activity levels of the 
  first and second subpopulations  respectively.  Denote by $\mu_1$ and $\mu_2$ their
  respective  relative size, suppose that the recovery rate $\gamma$ is equal to  $1$ for
  both subpopulations, and assume without loss of generality that $a\geq b$. In this case,
  we get that:
  \[ R_0 = a^2 \mu_1 + b^2 \mu_2, \quad R_e(\eta)=a^2 \eta_1\mu_1 + b^2 \eta_2\mu_2
    \quad\text{and}\quad
    C(\eta)=1- (\eta_1\mu_1 +  \eta_2\mu_2) 
  \]
  for the vaccination strategy $\eta=(\eta_1, \eta_2)$.  If $R_0 > 1$, then  the (unique)
  non-zero equilibrium  satisfies:
  \[ (1- \mathfrak{g}_i)\, \sum_{j=1}^2 K_{i,j}\, \mathfrak{g}_j\, \mu_j = \mathfrak{g}_i
  \quad\text{for $i=1, 2$}, \]
  and the corresponding vaccination strategy $\etae=1-  \mathfrak{g}$ is given by:
  \[ \etae= \left(  \frac{1}{1 + ac},  \frac{1}{1 + bc} \right), \]
  where $c\in [(1-R_0)/a, (1-R_0)/b]$ is  the unique positive solution of  the
  second-order
  equation $R_e(\etae)=1$.  It  is elementary to check that  in this case $C(\etae) <
  C(\etau)$, with an equality if and only if $a= b$.  However, the  critical  vaccination 
  strategy  with minimal  cost,  say  $\etao$, corresponds to  vaccinating in  priority
  the  population with  the highest activity rate, that is, if $a>b$:
  \[ \etao=\left(\frac{1- \min(1, b ^2 \mu_2)}{a^2 \mu_1}, \frac{1}{\max(1,b^2
  \mu_2)}\right).  \]
\end{example}

\begin{example}[Age and activity structure]
  In \cite{AMathematicalBritto2020}, Britton, Ball and Trapman study an SEIR model, where
  immunity can be obtained through infection. Using parameters derived from real-world
  data, these authors noticed that the disease-induced herd immunity level can, for some
  models, be substantially lower than the classical herd immunity threshold~$1 - 1/R_0$. 
  This can be reformulated in term of targeted vaccination strategies: prioritizing the
  individuals that are more likely to get infected in a SEIR epidemic may be more
  efficient than distributing uniformly the vaccine in the population.

  We use the same age and activity structures to determine which strategy between $\etae$
  and $\etauc$ is less costly. More precisely, the community is categorized into six age
  groups and contact rates between them are derived from an empirical study of social
  contacts \cite{UsingDataOnSWallin2006}.  For the  activity structure, individuals   are 
  categorized  into  three different activity  levels, which are arbitrary  and  chosen
  for illustration purposes: 50\%  of each age cohort have normal activity, 25\% have low
  activity corresponding to half as many contacts compared with normal activity, and 
  25\% have  high activity  corresponding to twice  as  many  contacts  as  normal
  activity.  Note that when the population is only structured by activity, the mixing is
  proportionate. Assuming  that the recovery rate is constant equal to $1$, we solved
  numerically Equation~\eqref{eq:endemic} and computed  in Table~\ref{tab:1} the cost of
  the  uniform and the equilibrium strategies for different values of $R_0$ and different
  population structures. In Table~\ref{tab:2}, the fractions of vaccinated individuals in
  the different age activity groups when following the strategy $\etae$. This is done by
  assuming $R_0 = 2$. Note than in this case, only three subpopulations need to be
  vaccinated at a level higher than~$1-1/R_0$. 
\end{example}

\newpage
   
\renewcommand{\arraystretch}{1.5}

\begin{table}[h]
\caption{Cost of the equilibrium vaccination compared to the herd immunity level for
different population structures. Numbers correspond to percentage.}
\begin{center}

\begin{tabular}{rSScSScSS}
    \toprule
 &\multicolumn{2}{c}{$R_0 = 2$} & \hspace{1.5em}& \multicolumn{2}{c}{$R_0 = 2.5$}
  &\hspace{1.5em}&\multicolumn{2}{c}{$R_0 = 3$}\\
 \cmidrule(r){2-3} \cmidrule(lr){5-6} \cmidrule(l){7-9}
     &{$C(\etae)$} & {$C(\etau)$} &&{$C(\etae)$} &{ $C(\etau)$ }&& {$C(\etae)$} & {$C(\etau)$}\\
    \midrule
  Homogeneous & 50 & 50 && 60 & 60 && 66.7 & 66.7\\
  Age structure & 46.6 & 50 && 56.7 & 60 && 63.9 & 66.7\\
  Activity structure & 40.1 & 50 && 50 & 60 && 57 & 66.7\\
  Age and activity structure & 35.7 & 50 && 45.2 & 60 && 52.2 & 66.7\\
    \bottomrule
\end{tabular}
\end{center}
\label{tab:1}
\end{table}
\bigskip
\begin{table}[h]
  \caption{Fraction of vaccinated individuals in different groups for the strategy~$\etae$.
    The population structure includes both age and
  activity. Numbers correspond to percentage.
  These values assume that $R_0 = 2$, so that the uniform critical vaccination
  consists in vaccinating $50\%$ of the population: 
  only three groups require more vaccine in the targeted strategy than in the uniform strategy.}
\begin{center}
\begin{tabular}{rSSS}
    \toprule
    Age group & {Low activity} & {Average activity} & {High activity} \\
    \midrule
    0--5 years & 12.0 & 21.4 & 35.3 \\
    6--12 years & 18.5 & 31.2 & 47.5 \\
    13--19 years & 22.9 & 37.3 & 54.3 \\
    20--39 years & 29.1 & 45.1 & 62.1 \\
  40--59 years &  20.9 & 34.6 & 51.4 \\
    $\geq$ 60 years & 12.4 & 22.1 & 36.2 \\
    \bottomrule
\end{tabular}
\end{center}
\label{tab:2}
\end{table}

\clearpage

\section{Proof}
\label{sec:main}

The differential equations governing the epidemic dynamics in meta-population~SIS models
were developed  in paper~\cite{lajmanovich1976deterministic} in finite dimension and
generalized in \cite{delmas_infinite-dimensional_2020}.

\subsection{The heterogeneous SIS model}

Let~$(\Omega, \cf,  \mu)$ be a  probability space, where~$x  \in \Omega$ represents  a
feature  and the  probability measure~$\mu(\mathrm{d}  x)$ represents the fraction of the
population with feature~$x$. The parameters of the SIS model are given by a \emph{recovery
rate function} $\gamma$, which is a positive bounded measurable function defined on
$\Omega$, and a \emph{transmission rate kernel} $k$, which is a non-negative measurable
function defined on $\Omega^2$. \medskip

In accordance with \cite{delmas_infinite-dimensional_2020}, we consider for a kernel $\kk$
on $\Omega$ and $q\in (1, +\infty )$ its   norm:
\begin{equation*}
  \norm{\kk}_{\infty ,q}=\sup_{x\in \Omega}\, \left(\int_\Omega \kk(x,y)^q\,
  \mu(\mathrm{d} y)\right)^{1/q}.
\end{equation*}
For a kernel $\kk$ on $\Omega$ such that $\norm{\kk}_{\infty ,q}$ is finite for some $q\in
(1, +\infty )$, we define the integral operator~$\Tinf_{\kk}$  on the set $\cl$ of bounded
measurable real-valued function on $\Omega$ by:
\begin{equation}\label{eq:def-Tkk}
  \Tinf_\kk (g) (x) = \int_\Omega \kk(x,y) g(y)\,\mu(\mathrm{d}y)
  \quad \text{for } g\in \cl \text{ and } x\in \Omega.
\end{equation}

By convention, for~$f,g$ two non-negative measurable functions defined on~$\Omega$
and~$\kk$ a kernel on~$\Omega$, we denote by $f\kk g$ the kernel on $\Omega$ defined by:
\begin{equation}
  \label{eq:def-fkg}
  f\kk g:(x,y)\mapsto f(x)\, \kk(x,y) g(y).
\end{equation}
We shall consider the kernel $\kkk=k/\gamma$ (corresponding to~$k\gamma^{-1}$, which 
differs in general from~$\gamma^{-1} k$), which is thus  defined by:
\begin{equation}
  \label{eq:def-kk}
  \boxed{    \kkk(x,y)=k(x,y)\, \gamma(y)^{-1}.}
\end{equation}
We shall assume that:
\begin{equation}\label{eq:bded}
  \boxed{ \norm{\kkk}_{\infty ,q}<\infty 
    \quad\text{for some $q\in (1, +\infty )$.}}
\end{equation}
The integral operator~$\Tinf_{\kkk}$  is   the  so  called  \emph{next-generation
operator}.

\medskip

Let $\Delta=\{f\in \mathscr{L}^\infty\,\colon\, 0\leq  f\leq 1\}$ be the
subset of non-negative functions bounded by~$1$, and let $\un\in \Delta$
be  the constant  function  equal  to 1.   The  SIS dynamics  considered
in~\cite{delmas_infinite-dimensional_2020} follows  the vector field~$F$
defined on~$\Delta$ by:
\begin{equation}\label{eq:vec-field}
\boxed{  F(g) = (\un - g) \Tinf_k (g) - \gamma g.}
\end{equation}
More precisely, we consider~$u=(u_t, t\in \R)$, where~$u_t\in \Delta$ for all~$t\in\R_+$
such that:
\begin{equation}\label{eq:SIS2}
 \boxed{   \partial_t u_t = F(u_t)} \quad\text{for } t\in \R_+,
\end{equation}
with initial condition~$u_0\in \Delta$. The value~$u_t(x)$ models the probability that an
individual of feature~$x$ is infected at time~$t$; it is proved
in~\cite{delmas_infinite-dimensional_2020} that such a solution~$u$ exists and is unique.

\medskip

An \emph{equilibrium} of~\eqref{eq:SIS2} is a function~$g \in \Delta$ such that~$F(g) =
0$. According to \cite{delmas_infinite-dimensional_2020}, there exists a maximal
equilibrium~$\mathfrak{g}$, \textit{i.e.}, an equilibrium such that all other
equilibria~$h\in \Delta$ are dominated by~$\mathfrak{g}$: $h \leq \mathfrak{g}$. It is to
this maximal equilibrium that the process stabilizes when started from a situation where
all the population is infected, that is,   if  $u_0=\un$, then we have:
\[
  \boxed{  \lim_{t\rightarrow \infty } u_t=\mathfrak{g}.}
\]
 
\medskip

For $T$ a bounded operator on $\cl$ endowed with its usual  supremum
norm, we define by~$\norm{T}_{\cl}$ its operator norm and its spectral
radius is given by:
\[
   \rho(T)=  \lim_{n\rightarrow \infty } \norm{T^n}_{\cl}^{1/n}.
\] 
The \emph{reproduction number}~$R_0$ associated to the SIS model given by~\eqref{eq:SIS2} is
the spectral radius of the next-generation operator:
\begin{equation}\label{eq:def-R0-2}
 \boxed{   R_0= \rho (\Tinf_{\kkk}),}
\end{equation}
If~$R_0\leq 1$ (sub-critical and critical case), then~$u_t$ converges pointwise to~$0$
when~$t\to\infty$. In particular, the maximal equilibrium~$\mathfrak{g}$ is equal to~$0$
everywhere. If~$R_0>1$ (super-critical case), then~$0$ is still an equilibrium but
different from the maximal equilibrium $\mathfrak{g}$, as~$\int_\Omega \mathfrak{g} \,
\mathrm{d}\mu > 0$.

\subsection{Vaccination strategies}\label{sec:vacc}

A \emph{vaccination strategy}~$\eta$ of a vaccine with perfect efficiency is an element
of~$\Delta$, where~$\eta(x)$ represents the proportion of \emph{\textbf{non-vaccinated}}
individuals with feature~$x$. Notice that~$\eta\, \mathrm{d} \mu$ corresponds in a sense
to the effective population.
In particular, the ``strategy'' that consists in vaccinating no one (resp. everybody)
corresponds to $\eta = \un$, the constant function equal to 1, (resp. $\eta
= \zero$, the constant function equal to 0).
\medskip

Recall the definition of the kernel~$f\kkk g$ from~\eqref{eq:def-fkg}. For~$\eta \in
\Delta$, the kernel~$\kk\eta=k\eta/\gamma$ has  finite norm $\norm{\cdot}_{\infty ,
q}$, so we can consider the bounded positive
operators~$\Tinf_{\kkk \eta }$ and~$\Tinf_{k\eta}$ on~$\mathscr{L}^\infty$. According
to \cite[Section~5.3.]{delmas_infinite-dimensional_2020}, the SIS equation with
vaccination strategy~$\eta$ is given by~\eqref{eq:SIS2}, where~$F$ is replaced by~$F_\eta$
defined by:
\begin{equation}\label{eq:vec-field-vaccin}
 \boxed{   F_\eta(g) = (\un - g) \Tinf_{k\eta}(g) - \gamma g.}
\end{equation}
We denote by~$u^\eta=(u^\eta_t, t\geq 0)$ the corresponding solution with initial
condition~$u_0^\eta \in \Delta$. We recall that~$u_t^\eta(x)$ represents the probability
for an non-vaccinated individual of feature~$x$ to be infected at time
$t$.
We define the \emph{effective reproduction number} $R_e(\eta)$
associated to the vaccination strategy $\eta$ as  the spectral radius
of the effective next-generation operator~$\Tinf_{\kkk \eta}$:
\begin{equation}\label{eq:def-R_e}
  \boxed{R_e(\eta)=\rho(\Tinf_{\kkk\eta}).}
\end{equation}
For example, for the trivial vaccination strategies we get~$R_e(\un) =
R_0$ and $R_e(\zero) = 0$.

We denote by~$\mathfrak{g}_\eta$ the corresponding maximal equilibrium. In particular, we have:
\begin{equation}
  \label{eq:F(g)=0}
  F_\eta(\mathfrak{g}_\eta)=0.
\end{equation}
In particular, we have:
\[
  R_e(\un)=R_0 \quad\text{and}\quad \mathfrak{g}=\mathfrak{g}_{\un}.
\]

\subsection{Critical vaccination strategies}

If~$R_0\geq 1$, then a  vaccination
strategy~$\eta$ is called \emph{critical} if it achieves precisely herd immunity, that
is~$R_e(\eta) = 1$.

As the spectral radius is positively
homogeneous (that is, $\rho(\lambda t)=\lambda \rho(T)$ for $\lambda\geq 0$),
we also get, when~$R_0\geq 1$, that the uniform strategy that corresponds to the constant
function:
\[
\boxed{\etauc = \frac{1}{R_0}\un}
\]
is critical, as~$R_e(\etauc)=1$. This is consistent with results obtained in the
homogeneous model.

\medskip

As hinted in \cite[Section~4.5]{hethcote} for vaccination control of gonorrhea, it is
interesting to consider vaccinating people with feature $x$ with
probability~$\mathfrak{g}(x)$. This corresponds to the strategy based on the maximal
equilibrium:
\[
 \boxed{  \etae = \un -\mathfrak{g}.}
\]
The following  result entails  that this strategy  is critical  and thus
achieves  herd  immunity.  Recall  that  in the  (infinite
dimensional)  SIS  model~\eqref{eq:vec-field}  on the  probability  space
$(\Omega, \cf, \mu)$ the recovery rate function $\gamma$ is positive and
bounded,  the  transmission  rate  $k$  is  non-negative  and  the  norm
$\norm{\kkk}_{\infty ,  q} $ of  the kernel $\kkk=k/\gamma$ is  finite for
some $q\in (1, +\infty )$.

\begin{theorem}[The maximal equilibrium yields a critical vaccination]
  \label{theo:main}
  Consider  the  SIS  model~\eqref{eq:vec-field} under  the  boundedness
  assumption~\eqref{eq:bded}.  If~$R_0\geq  1$,   then  the  vaccination
  strategy $\etae$ is critical, that is,~$R_e(\etae)=1$.
\end{theorem}

This result will be proved below as a part of Proposition~\ref{prop:caract-g}.

\subsection{Proof of Theorem~\ref{theo:main}}\label{sec:proof}

For an
operator $A$, we denote by $A^\top$ its adjoint. We first give a preliminary
lemma. For the convenience of the reader, we only use references to the results recalled in
\cite{delmas_infinite-dimensional_2020} for positive operators on Banach
spaces. In particular,  we shall use that
if  $\kk$ and $\kk'$ are two  (non-negative) kernels on $\Omega$ with
finite norms $\norm{\cdot}_{\infty ,q}$ for some $q\in (1, +\infty
)$, then we have that:
\begin{equation}
  \label{eq:k>k}
  \kk\geq  \kk' \,\Longrightarrow\,  \rho(\Tinf_\kk)\geq \rho(\Tinf_{\kk'}),
\end{equation}
see for example
\cite[Theorem~3.5(i)]{delmas_infinite-dimensional_2020} as the operator
$\Tinf_\kk-\Tinf_{\kk'}$ is positive. 
We shall also used
that for two bounded operators $T$ and $S$ on $\cl$:
\begin{equation}
  \label{eq:AB=BA}
  \rho(TS)=\rho(ST).
\end{equation} 

We first state two technical lemmas. 

\begin{lemma}\label{lem:prelim-result}
  Let   $\kk$  be   a  non-negative   kernel  on   $\Omega$  such   that
  $\norm{\kk}_{\infty ,  q} $ is finite  for some $q\in (1,  +\infty )$
  and   and consider  the  positive
  bounded linear integral operator $\Tinf_{\kk}$ on $\cl$. If there
  exists        $g\in         \mathscr{L}^\infty_+        $,        with
  $\int_\Omega g \, \mathrm{d} \mu >0$ and $\lambda > 0$ satisfying:
  \[
    \Tinf_{\kk}(g)(x) > \lambda g(x), \quad \text{for all} \quad x \text{ such that }
  g(x)>0, \]
  then we have $\rho(\Tinf_{\kk})>\lambda$.
\end{lemma}

\begin{proof}
  We simply write $\Tinf$ for $\Tinf_{\kk}$.  Let $A =\set{ g>0}$ be the
  support of  the function  $g$.  Let $\Tinf'$  be the  bounded operator
  defined    by   $\Tinf'(f)    =\ind{A}   \Tinf(\ind{A}    f)$.   Since
  $\Tinf'(g)= \ind{A} \Tinf(\ind{A} g) =  \ind{A} \Tinf(g) > \lambda g$,
  we     deduce    from     the     Collatz-Wielandt    formula
  \cite[Proposition~3.6]{delmas_infinite-dimensional_2020}   that
  $\rho(\Tinf')        \geq        \lambda>0$.       According        to
  \cite[Lemma~3.7~(v)]{delmas_infinite-dimensional_2020},  there  exists
  $v\in L^q_+ \setminus \{ 0 \}$,  seen as an element of the topological
  dual of $\cl$,  a left Perron eigenfunction of $\Tinf'$,  that is such
  that  $  (\Tinf')^\top(v) =\rho(\Tinf')  v$.  In  particular, we  have
  $v=  \ind{A} \,  v$  and thus  $\int_A  v \,  \mathrm{d}\mu  > 0$  and
  $\int_\Omega vg\, \mathrm{d}\mu > 0$. We obtain:
  \[
    (\rho(\Tinf') - \lambda) \braket{v,g} = \braket{v, \Tinf'(g) -
    \lambda g} > 0.
  \]
 As $\Tinf'=\Tinf_{\kk'}$ with $\kk'= \ind{A} \kk\ind {A}\leq  \kk$, we deduce
 from~\eqref{eq:k>k} that $\rho(\Tinf) \geq \rho(\Tinf') >
  \lambda$.
\end{proof}

\begin{lemma}
  \label{lem:Fh>0}
 Consider the  SIS model~\eqref{eq:vec-field} under the boundedness assumption~\eqref{eq:bded}. 
 Let $\eta, g \in \Delta$. If $F_\eta(g) \geq 0$, then we have $g\leq
  \mathfrak{g}_\eta$.
\end{lemma}
\begin{proof}
 Consider the solution $u_t$ of the SIS model
  $\partial_t u_t = F_\eta(u_t)$  with vaccination $\eta$ and initial
  condition $u_0 = g$.  According to
  \cite[Proposition~2.10]{delmas_infinite-dimensional_2020}, this
  solution is non-decreasing since $F_\eta(g) \geq 0$. According to
  \cite[Proposition~2.13]{delmas_infinite-dimensional_2020}, the
  pointwise limit of $u_t$ is an equilibrium. As this limit is
  dominated by the maximal equilibrium $\mathfrak{g}_\eta$ and since
  $u_t$ is non-decreasing, this proves that $g\leq\mathfrak{g}_\eta$.
\end{proof}

The next result characterizes the maximal
equilibrium~$\mathfrak{g}$ among all equilibria by various spectral
properties; 
Theorem~\ref{theo:main} may be viewed as a corollary to this characterization.
Recall
that $R_0=R_e(\un)$,  and that the vector field~$F$ is defined by \eqref{eq:vec-field}. Let
$DF[h]$ denote the bounded linear operator on $\cl$ of the derivative
of the map $f \mapsto F(f)$ defined on $\cl$ at point $h$:
\[
  DF[h](g) = (1-h) \Tinf_k(g) - (\gamma + \Tinf_{k}(h))g
  \quad \text{for $h,g\in \cl$.}
\]

Let $s(A)$ denote the spectral bound of the bounded operator $A$, see Equation~(33)
in~\cite{delmas_infinite-dimensional_2020}.

\begin{proposition}[Equivalent conditions for maximality]
  \label{prop:caract-g}
Consider the  SIS model~\eqref{eq:vec-field} under the boundedness assumption~\eqref{eq:bded}. Let $h$ in
  $\Delta$ be an equilibrium, that is, $F(h) = 0$. The following properties are equivalent:
  \begin{enumerate}[(i)]
  \item\label{h=g} $h=\mathfrak{g}$,
  \item\label{s(g)<0} $s(DF[h]) \leq 0$,
  \item\label{Rh2<1} $R_e((1-h)^2)\leq 1$.
    \item\label{s(1-h)<0} $s(DF_{(1-h)}[\zero]) \leq 0$.
  \item\label{Rh<1} $R_e(1-h)\leq 1$.
  \end{enumerate}
  Furthermore,  $\mathfrak{g}=\zero$ if and only if $ R_0\leq 1$, and
  if $\mathfrak{g}\neq \zero$, then it is critical: $R_e(1-\mathfrak{g})=1$.
\end{proposition}

\begin{remark}[On stability]\label{rem:stab}
  From a dynamical systems point of view, this proposition links together two different
  stability properties.  The (classically equivalent) conditions~\ref{s(g)<0}
  and~\ref{Rh2<1} state that for the original dynamics given by~\eqref{eq:SIS2} with
  vector field~$F$, the equilibrium $h$ is not linearly unstable.  Similarly,
  conditions~\ref{s(1-h)<0} and \ref{Rh<1} both state that in the vaccinated dynamics
  given by the modified vector field $F_{1-h}$ defined by~\eqref{eq:vec-field-vaccin}, the
  disease-free equilibrium $0$ is not linearly unstable.

  In particular, in the original dynamics given by~\eqref{eq:SIS2}, equilibria that are
  not maximal are necessarily linearly unstable.
\end{remark}

\begin{proof}
  Let $h \in \Delta$ be an equilibrium, that is $F(h)=0$.

  \medskip 

  The equivalence between~\ref{s(1-h)<0} and \ref{Rh<1} is a direct consequence
  of~\cite[Proposition~4.2]{delmas_infinite-dimensional_2020}.
  
  Let us show the equivalence between \ref{s(g)<0} and \ref{Rh2<1}. According to
  the same~\cite[Proposition~4.2]{delmas_infinite-dimensional_2020}, $s(DF[ h]) \leq 0$ if and only
  if:
  \[
    \rho\left(\Tinf_\kk\right) \leq 1 \quad \text{with} \quad
    \kk(x,y) = (1 - h(x))\frac{ k(x,y) }{\gamma(y) + \Tinf_k(h)(y)}\cdot
  \]
  Since $F(h)=0$, we have $(1-h)/\gamma= 1/(\gamma+ \Tinf_k(h))$. This gives:
  \begin{equation}\label{eq:Tk-Fh}
    \mathsf{k}(x,y) = (1 - h(x))\frac{ k(x,y) (1-h(y))}{\gamma(y)}
  \end{equation}
  and thus $\Tinf_\kk = M_{1-h}\, \Tinf_{k/\gamma} \, M_{1-h}$, where $M_f$ is the
  multiplication operator by $f$. Recall the definition~\eqref{eq:def-R_e} of $R_e$.
 We deduce from~\eqref{eq:AB=BA} that:
  \begin{equation}\label{eq:commute} \rho\left(\Tinf_\kk\right) =
    \rho\left(\Tinf_{k/\gamma} M_{(1-h)^2} \right) = R_e((1-h)^2).
  \end{equation}
  This gives the equivalence between \ref{s(g)<0} and \ref{Rh2<1}.

  \medskip

  We prove that \ref{h=g} implies \ref{Rh<1}. Suppose that $R_e(1-h)>1$.  Thanks
  to~~\eqref{eq:AB=BA}, 
  we have $\rho(M_{1-h} \Tinf_{k/\gamma})=
  \rho(\Tinf_{k/\gamma}M_{1-h})= R_e(1-h)>1$. According to
  \cite[Lemma~3.7~(v)]{delmas_infinite-dimensional_2020}, there exists $v\in L^q_+
  \setminus \{ 0 \}$ a left Perron eigenfunction of $\Tinf_{(1-h)k/\gamma}$, that is $
  \Tinf_{(1-h)k/\gamma}^\top(v) = R_e(1-h) v$. Using $F(h)=0$, and thus
  $(1-h) \Tinf_k(h)=\gamma h$, for the last equality, we have:
  \begin{equation*}
    R_e(1-h) \braket{v,\gamma h} = \braket{v,
    (1-h)\Tinf_{k/\gamma}(\gamma h)} = \braket{v, \gamma h}.
  \end{equation*}
  We get $\braket{v,\gamma h}=0 $ and thus $\braket{v,\ind{A}} = 0$, where
  $A=\set{h>0}$ denote the support of the function $h$. Since $
  \Tinf_{(1-h)k/\gamma}^\top(v) =
  R_e(1-h) v$ and setting $v'=(1-h)v$ (so that $v'=v$ $\mu$-almost surely on $A^c$), we
  deduce that:
  \begin{equation*}
    \Tinf_{k'/\gamma}^\top(v') = R_e(1-h) v',
  \end{equation*}
  where    $k'=\ind{A^c}\,   k\,    \ind{A^c}$.   This    implies   that
  $\rho(  \Tinf_{k'/\gamma})\geq   R_e(1-h)$.  Since   $k'=(1-h)k'$  and
  $k-k'\geq 0$,  we get that $\Tinf_{k/\gamma}-  \Tinf_{k'/\gamma}$ is a
  positive          operator.    Using~\eqref{eq:k>k} for the
  inequality as $(1-h) k'/\gamma\leq  (1-h) k/\gamma$, we          deduce          that
  $ \rho(\Tinf_{k'/\gamma})= \rho ( M_{1-h} \Tinf_{k'/\gamma}) \leq \rho
  (  M_{1-h}\Tinf_{k/\gamma})= R_e(1-h)$.   Thus,  the
  spectral    radius     of    $\Tinf_{k'/\gamma}$    is     equal    to
  $R_e(1-h)$.                        According                        to
  \cite[Proposition~4.2]{delmas_infinite-dimensional_2020},        since
  $\rho(\Tinf_{k'/\gamma})>1$,                there               exists
  $w  \in \mathscr{L}^\infty_+\setminus  \{ 0  \}$ and  $\lambda>0$ such
  that:
  \begin{equation*}
    \Tinf_{k'}(w) - \gamma w = \lambda w.
  \end{equation*}
  This also implies that $w=0$ on $A=\set{h>0}$, that is $wh=0$ and thus $w \Tinf_k(h)=0$
  as $\Tinf_k (h) =\gamma h/(1-h)$. Using that $F(h)=0$,
  $\Tinf_k(w)=\Tinf_{k'}(w)=(\gamma+\lambda) w$ and $h\Tinf_k(w)=0$, we obtain:
  \begin{equation*}
    F(h+w) = w(\lambda - \Tinf_k(w)) .
  \end{equation*}
  Taking $\varepsilon>0$ small enough so that $\varepsilon \Tinf_{k }(w) \leq \lambda/2$
  and $\varepsilon w \leq 1$, we get $h+\varepsilon w\in
  \Delta$ and $ F( h +\varepsilon
  w) \geq 0$. Then use
  Lemma \ref{lem:Fh>0} to deduce that $h+\varepsilon w \leq \mathfrak{g}$ and thus
  $h \neq \mathfrak{g} $.

  \medskip

  To   see   that   \ref{Rh<1}    implies   \ref{Rh2<1},   notice   that
  $(1-h)\geq  (1-h)^2$,  and  then  use~\eqref{eq:k>k}  to  deduce  that
  $\rho(\Tinf_{\kkk  (1-h)})\geq  \rho(\Tinf_{\kkk (1-h)^2})$  and  thus
  $ R_e(1-h)\geq R_e((1-h)^2)$.

  \medskip

  We prove that \ref{Rh2<1} implies \ref{h=g}. Notice that $F(g)=0$ and $g\in \Delta$
  implies that $g<1$. Assume that $h \neq \mathfrak{g}$. Notice that $\gamma/(1-h)= \gamma
  + \Tinf_{k}(h)$, so that $\gamma (\mathfrak{g} - h)/(1 - h) \in \cl_+$. An elementary
  computation, using $F(h)=F(\mathfrak{g})=0$ and $\kk$ defined in~\eqref{eq:Tk-Fh},
  gives:
  \[
    \Tinf_\kk \left( \gamma \frac{\mathfrak{g} - h}{1 - h}\right)
    = (1-h)\Tinf_{k}\left(\mathfrak{g} - h \right)
    = \gamma
    \frac{\mathfrak{g} - h}{1 - \mathfrak{g}}
    =\frac{1-h}{1 - \mathfrak{g}} \, \gamma
    \frac{\mathfrak{g} - h}{1 - h} \cdot
  \]
  Since $h \neq \mathfrak{g}$ and $h\leq \mathfrak{g}$, we deduce that $(1-h)/ (1 -
  \mathfrak{g})\geq 1$, with strict inequality on $\set{ \mathfrak{g}-h>0}$ which is a set
  of positive measure. We deduce from Lemma~\ref{lem:prelim-result} (with $k$ replaced by
  $\kk \gamma$) that $\rho\left(\Tinf_{\mathsf{k}}\right)>1$. Then use~\eqref{eq:commute}
  to conclude.

  \medskip

  To conclude notice that $\mathfrak{g}=0 \Longleftrightarrow R_0\leq 1$ is a consequence
  of the equivalence between \ref{h=g} and \ref{Rh<1} with $h=0$ and $R_0=R_e(\un)$.

  Using that $F(\mathfrak{g})=0$, we get $\Tinf_k(\mathfrak{g})= \gamma
  \mathfrak{g}/(1-\mathfrak{g})$. We deduce that
  $\Tinf_{k(1-\mathfrak{g})/\gamma}(\Tinf_k(\mathfrak{g})) = \Tinf_k(\mathfrak{g})$. If
  $\mathfrak{g}\neq 0$, we get $\Tinf_k(\mathfrak{g})\neq 0$ (on a set of positive
  $\mu$-measure). This implies that $R_e(1-\mathfrak{g})\geq 1$. Then use \ref{Rh<1} to
  deduce that $R_e(1-\mathfrak{g})=1$ if $\mathfrak{g}\neq 0$.
\end{proof}

\section*{Acknowledgments}

This work is partially supported by Labex Bézout reference ANR-10-LABX-58.

\printbibliography

\end{document}